\documentclass{amsart}[12pt]

\usepackage{tedmath}
\usepackage{permImages}
\usepackage{numberingpaperted}

\usepackage{tikz}
\usetikzlibrary{arrows,shapes}

\usepackage[bookmarks]{hyperref}
\hypersetup{colorlinks,
citecolor=black,
filecolor=black,
linkcolor=black,
urlcolor=black,
pdftex}
\usepackage{cite}

\usepackage{graphicx}
\usepackage{caption}
\usepackage{subcaption}

\def\sq{\square}

\def\rr{\mathbb R}

\def\si{\sigma}

\def\<{\langle}
\def\>{\rangle}

\def\Bax{\mathcal{B}}

\def\.{\hskip.04cm}
\def\ts{\hskip.02cm}

\title[Shape of Doubly Alternating Baxter Permutations]
{The Expected Shape of Random Doubly Alternating Baxter Permutations}
\author[Theodore Dokos and Igor Pak]{Theodore Dokos$^\star$ and Igor Pak$^\star$}

\thanks{\thinspace ${\hspace{-.45ex}}^\star$
Department of Mathematics,
UCLA, Los Angeles, CA, 90095.
\hskip.06cm
Email:
\hskip.02cm
\texttt{\{tdokos,pak\}@math.ucla.edu}}

\theoremstyle{definition}
\newtheorem{remark}[theorem]{Remark}

\newcommand{\Psum}[1]{P_{\textsc{\tiny #1}}}

\begin{document}
\maketitle{}


\begin{abstract}
Guibert and Linusson introduced in~\cite{guibert-linusson} the family of
\emph{doubly alternating Baxter permutations}, i.e.~Baxter permutations
$\si \in S_n$, such that $\si$ and $\si^{-1}$ are alternating.  They proved
that the number of such permutations in~$S_{2n}$ and~$S_{2n+1}$
is the Catalan number~$C_n$.  In this paper we explore the expected limit shape
of such permutations, following the approach by Miner and Pak~\cite{miner-pak}.
\end{abstract}

\vskip1.4cm

\section{Introduction}

\noindent
A \emph{Catalan structure} is a family of combinatorial objects whose number is the Catalan number
$$
C_n \, = \, \frac{1}{n+1}\ts \binom{2n}{n}\..
$$
There is a staggering amount of literature on various Catalan structures,
the list of which is ever growing (see~\cite{Gould,Pak,Slo,S1,S2}).
This multitude and diversity of Catalan structures 
is in part a
consequence of their different nature, and in part a misperception,
as many such structures are essentially equivalent, via a ``nice bijection''.
In the latter case, this can happen despite the apparently different geometric
representations of the Catalan objects (think polygon triangulations
vs.~binary trees). Discerning different combinatorial structures
can be difficult and even harder to formalize (as in, how do you prove
non-existence of a ``nice bijection''?).

\smallskip

In this paper we study the asymptotic behavior of doubly alternating
Baxter permutations $\si\in S_n$.  Following~\cite{miner-pak},
we compute the \emph{expected limit shape} of $\Bax_n$ by viewing permutations
as 0-1 matrices, and averaging them.  When scaled appropriately
the resulting distribution on $[0,1]^2$ converges to a \emph{limit
surface} $\Phi(x,y)$ with an explicit formula that seems new in
the literature, thus differentiating this Catalan structure
from others (see Theorem~\ref{thm:asymptotic} below).

\smallskip

\emph{Baxter permutations} \ts are defined to be $\si\in S_n$,
such that there are no indices \ts $1\le i < j < k\le n$, which
satisfy \ts $\si(j + 1) < \si(i) < \si(k) < \si(j)$ \ts or \ts
$\si(j) < \si(k) < \si(i) < \si(j + 1)$.  They were introduced
and enumerated by Baxter in 1964~\cite{baxter} (see also~\cite{chung-et,Vie}),
and recently became a popular subject (see~$\S\ref{s:fin-bax}$).

The \emph{alternating permutations} are defined to be $\si\in S_n$,
such that $\si(1)<\si(2)>\si(3)<\si(4)>\ldots$.  They were defined
by Andr\'{e} in~1879, and extensively studied over the years
(see~\cite{Sta,S1}).  Their number is known as the
\emph{Euler--Bernoulli number} with the exponential
generating function $\tan(x) + \sec(x)$, and they have
numerous connections to other fields.

A Baxter permutation $\si \in S_n$ is called
\emph{doubly alternating} if both~$\si$ and $\si^{-1}$ are alter-nating.
Guibert and Linusson showed in~\cite{guibert-linusson} that
the set $\Bax_n$ of such permutations is a Catalan structure:
$$\bigl|\Bax_{2n}\bigr| \. = \. \bigl|\Bax_{2n+1}\bigr| \. = \. C_n\,.$$
Let $P(m,i,j)$ denote the probability that a random $\si\in \Bax_{2m}$ has
$\si(i)=j$.  Here is our main result.

\medskip

\begin{theorem}\label{thm:asymptotic}
  Let $0 < \alpha < \beta < 1 - \alpha$.
  Then, as $m$ goes to infinity, we have:
  \begin{equation*}
    P(m,\lfloor 2\alpha m \rfloor , \lfloor 2\beta m \rfloor) \, \sim \, \frac{\varphi(\alpha,\beta)}{m}\.,
  \end{equation*}
  where
  \begin{equation*}
    \varphi(\alpha,\beta) \, = \.\frac{1}{8\pi}\. \int_0^\alpha \int_0^{\alpha-y} \.\frac{dx\ dy}{[(x+y)(\beta-x)(1-\beta-y)]^{3/2}}\..
  \end{equation*}
\end{theorem}

\begin{figure}[h]
  \centering
  \includegraphics[width=8cm]{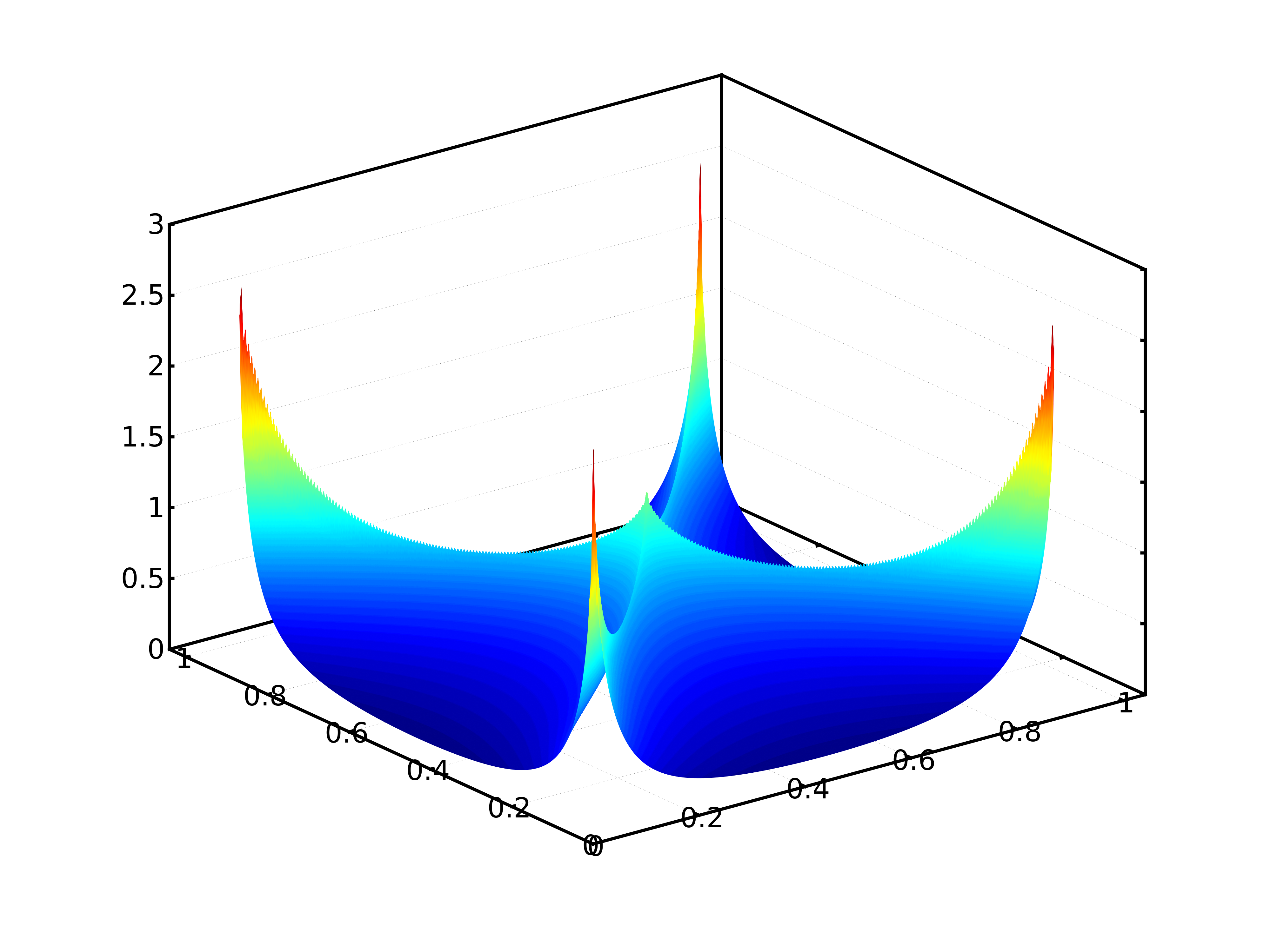}
  \caption{The graph of the limit surface~$\Phi(x,y)$.}
  \label{fig:limit-surface}
\end{figure}

In other words, when the distribution $P(m,i,j)$ is scaled appropriately,
it converges to the limit surface $\Phi(x,y)$ as shown in Figure~\ref{fig:limit-surface}.
Here the surface $\Phi: (0,1)^2\to \rr$ is obtained from
$\varphi: \{0<x<y<1, \. x+y < 1\}\to \rr$ by reflection across both diagonals.

There are a few things to note.  First, the surface $\Phi(x,y)$ has the symmetry
group of a square. Curiously, most of symmetries do not appear in the actual
numbers~$B(m,i,j)$.
We discuss this in further detail in Section \ref{sec:final-remarks}.
Second, most of the surface is ``under water'', i.e. we have $\Phi(x,y)<1$
everywhere except near the corners and at the center spike $\Phi(\frac12,\frac12)=\frac32$.
The reason is simple: the peaks in the corners actually diverge
(see Theorem~\ref{t:spikes}).\footnote{Our graph in Figure~\ref{fig:limit-surface}
has been truncated.}   Finally, the limit shape for $\Bax_n$ is fundamentally different
from those in~\cite{miner-pak}, where the limit surfaces are degenerate.
We give more details on these in Section~\ref{sec:final-remarks}.

\smallskip

The rest of the paper is structured as follows. In the next section
we introduce the basic notations.  In Section~\ref{sec:DAB-results},
we present the main lemma (Lemma~\ref{mainlemma:main-theorem-enumerative}),
giving an explicit triple summation formula for~$P(m,i,j)$.
Section~\ref{sec:proof-mainlemma} contains the proof of the main lemma.
We then use the main lemma to prove Theorem~\ref{thm:asymptotic} in
Section~\ref{sec:asymptotics}.  We conclude with final remarks and
open problems in Section~\ref{sec:final-remarks}.

\bigskip

\section{Basic definitions and notation}

\noindent
For a permutation $\sigma \in S_n$, its \emph{complement} $\sigma^c$
is defined pointwise by $\sigma^c (i) = n + 1 - \sigma (i)$.
Given $\sigma = a_1 \cdots a_n \in S_n$ and permutations
$\tau_1, \dots ,\tau_n$, we define the \emph{inflation}
$\sigma[\tau_1, \dots ,\tau_n]$ to be the permutation obtained
by replacing $a_i$ with a copy of $\tau_i$, shifted to be in
the same relative position as $a_i$.

We say that $a_n \sim b_n$, or $a_n$ is asymptotically equivalent
to~$b_n$ if  $a_n/b_n \rightarrow 1$ as $n\rightarrow \infty$.
Recall Stirling's formula,
\begin{equation*}
  n! \. \sim \. \sqrt{2\ts \pi\ts n} \left( \frac{n}{e} \right)^n\..
\end{equation*}
It gives the following asymptotic formula for the Catalan numbers:
\begin{equation}\label{eq:catalan-asymptotic}
  C_n \. \sim \. \frac{1}{\sqrt{\pi} \ts n^{3/2}} \,\ts 4^n\,.
\end{equation}
In Section~\ref{sec:asymptotics}, we will make heavy use of this approximation.

\bigskip

\section{Main Lemma}
\label{sec:DAB-results}

\noindent
As in the introduction, let $\mathcal{B}_n$ denote the set of doubly
alternating Baxter permutations of length~$n$.  Recall that $|\mathcal{B}_{2m}|$
is equal to the $m$-th Catalan number $C_m$.  It was proved
in~\cite{guibert-linusson} that for $\sigma \in \mathcal{B}_n$ with $n$~odd,
we have $\sigma = 12[1,\tau^c]$, for some $\tau \in \mathcal{B}_{n-1}$.
Therefore, the limit shape of $\mathcal{B}_{2m+1}$ is that of~$\mathcal{B}_{2m}$.
From this point on, we restrict our attention to sets~$\mathcal{B}_{2m}$.

It was further proved~\cite{guibert-linusson}, that elements of~$\mathcal{B}_{2m}$
can be described by a standard recursive Catalan structure.
Specifically, if $\sigma(1) = 2k+1$ (and $\sigma(1)$ must always be odd),
then we have $\sigma = 2341[1,\tau^c,1,\omega]$ with $\omega \in \mathcal{B}_{2k}$
and $\tau \in \mathcal{B}_{2(m-k-1)}$.
Figure~\ref{fig:dab-recursive-structure} illustrates this decomposition.
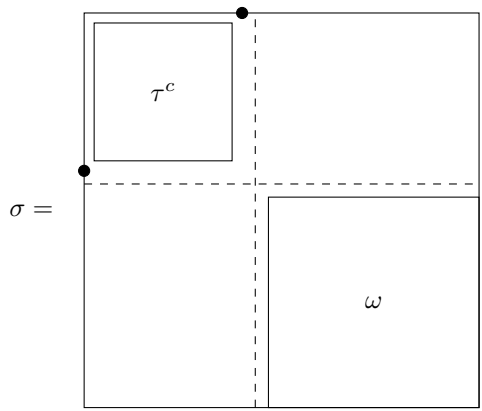
\begin{figure}[h]\centering
  \begin{tikzpicture}
    [scale=0.35,
    every node/.style={align=center}]

    \draw (0,7.5) node {$\sigma = $};

    \begin{scope}[xshift = 2cm]
      \draw (0,0) rectangle (15,15);

      \filldraw (0,9) circle (6pt);
      \filldraw (6,15) circle (6pt);
      \draw (3,12) node {$\tau^c$};
      \draw (0.38,9.38) rectangle (5.62,14.62);

      \draw[dashed] (0,8.5) -- (15,8.5);
      \draw[dashed] (6.5,0) -- (6.5,15);

      \draw (7,0) rectangle (15,8);
      \draw (11,4) node {$\omega$};
    \end{scope}
  \end{tikzpicture}
  \caption{A recursive description for $\sigma \in \mathcal{B}_{2m}$.}
  \label{fig:dab-recursive-structure}
\end{figure}

For $r,s \geq 1$, let $p_{r,s}$ be the number of Dyck paths of
length~$2(r+s-1)$, which first revisit the horizontal axis by
step~$2s$ at the latest.  Then $p_{r,s}$ is given by a partial
Catalan convolution:
\begin{equation*}
  p_{r,s} \, = \, \sum_{k=1}^{s} \. C_{r+s-(k+1)} \cdot C_{k-1}
  \, = \, C_{r+s-2}C_0 \. + \. \ldots \. + \. C_{r-1}C_{s-1}\..
\end{equation*}
Let $B(m,i,j)$ be the number of permutations $\sigma \in \mathcal{B}_{2m}$
such that $\sigma(i) = j$.
%
We use the recursive structure above to obtain the following summation
formula for~$B(m,i,j)$, in terms of the Catalan numbers.

\smallskip

\begin{mainlemma}\label{mainlemma:main-theorem-enumerative}
  Let $B(m,i,j)$ and $p_{r,s}$ be defined as above, and let
  $a = \left\lceil i/2 \right\rceil$, $b = \left\lfloor j/2 \right\rfloor$.
  Suppose further that $i \leq j \leq 2m-i$.  Then we have:
  \begin{equation}
    B(m,i,j) = \left\{
      \begin{array}{rll}
        C_b C_{m-b-1} \!\!\!\! & + \displaystyle\sum_{r=1}^{a - 1} \displaystyle\sum_{s=1}^{a - r} p_{r,s} \cdot C_{b-r+1} \cdot C_{m-b-s-1} & \text{for $j$ odd}\.,\\
        [2em]
        & \phantom{+} \displaystyle\sum_{r=1}^{a - 1} \displaystyle\sum_{s=1}^{a - r} p_{r,s} \cdot C_{b-r} \cdot C_{m-b-s} & \text{for $j$ even}\..
      \end{array}\right. \label{eq:main-lemma-equation}\tag{$\ast$}
  \end{equation}
\end{mainlemma}

\smallskip

We prove the main lemma by induction in the next section, and use it in
Section~\ref{sec:asymptotics} to prove Theorem~\ref{thm:asymptotic}.

\bigskip

\section{Proof of Main Lemma}
\label{sec:proof-mainlemma}

\subsection{The symmetries}
We observe the following easy properties of~$B(m,i,j)$.
\begin{proposition}
  We have
  \begin{enumerate}
  \item[$(i)$]
    $B(m,i,j) = B(m,j,i) = B(m,2m+1-j,2m+1-i)$\ts.

  \item[$(ii)$]
    For~$i=1$, we have
    \begin{equation*}
      B(m,1,j) = \left\{
        \begin{array}{ll}
          0 & \text{if $j$ is even}\.,\\
          C_{b} \cdot C_{m- b - 1} & \text{if $j = 2b+1$ is odd}\..
        \end{array}
      \right.
    \end{equation*}

  \item[$(iii)$]
    Let $i>1$ and $j < 2m$.
    Then
    \begin{equation*}\label{eq:basic-recurrence-for-dabs}
      \begin{aligned}
        B(m,i,j)\, = \, & \sum_{k=1}^{\lceil i/2 \rceil - 1} C_{k-1} \cdot B(m-k,i-2k,j) \,
        + \ts \sum_{k=m - \lfloor j/2 \rfloor + 1}^{m} C_{m-k} \cdot B(k-1,i-1,2m-j)\ts.
      \end{aligned}
  \end{equation*}

  \end{enumerate}
\end{proposition}

\begin{proof}
    The equalities in~$(i)$ correspond to reflection over the main diagonal and main antidiagonal, respectively.
    These reflections preserve both the Baxter and doubly alternating properties.

    The equality in~$(ii)$ follows immediately from the recursive structure of $\mathcal{B}_{2m}$ given in Figure~\ref{fig:dab-recursive-structure}.
    As a special case, note that $B(m,1,1) = C_{m-1}$\ts.

    The recurrence relation in~$(iii)$ follows from conditioning on $\sigma(2k) = 2m$, and checking the resulting restrictions
    on $\sigma(i)=j$.
    Note that for the second sub-sum, we require $2(k-1) \geq 2m -j$ to satisfy the required constraints.
\end{proof}

The first part of this proposition shows that our Main Lemma is sufficient to calculate $B(m,i,j)$ for \emph{all} $i,j$, by using reflections.

\subsection{Preliminary calculations}
Throughout the rest of this section, we assume that $i \leq j \leq 2m-i$.
Let us further refine the recurrence in equation in~$(iii)$.
We need to ensure that all terms $B(n,p,q)$ in the above satisfy the inequalities $p \leq q \leq 2n -p$.
The first sum requires no modification.
The second sum must be split however, and a symmetry applied to one group of terms.
Since we need $i-1 \leq 2m-j \leq 2k-i-1,$ it follows that the terms with 
$k \geq m + \frac{1}{2}(i+1-j)$ are already of the desired form.
For summands preceding this cut-off point, we will apply an antidiagonal reflection to the terms $B(*,*,*)$.
For notational convenience, let 
$h = \left\lfloor \frac{1}{2} (j - i - 1) \right\rfloor$ and let $a,$ $b$ be as in the Main Lemma.
Then our triple sum becomes
\begin{align*}
  B(m,i,j) =& \sum_{k=1}^{a - 1} C_{k-1} \cdot B(m-k,i-2k,j)\\
  &+ \sum_{k=m - b +1}^{m - h - 1} C_{m-k} \cdot B(k-1,j-1-2(m-k),2k-i)\\
  &+ \sum_{k=m - h}^{m} C_{m-k} \cdot B(k-1,i-1,2m-j).
\end{align*}
We reindex the second two sums of the above, by letting $k \gets m - b + \ell$:
\begin{align}\label{eq:triple-sum-reindexed}
  B(m,i,j) =& \sum_{k=1}^{a - 1} C_{k-1} \cdot B(m-k,i-2k,j)\nonumber\\
  &+ \sum_{\ell = 1}^{b - h - 1} C_{b - \ell} \cdot B(m - b + \ell-1,j-1-2(b - \ell),2(m - b + \ell)-i)\\
  &+ \sum_{\ell = b - h}^{b} C_{b - \ell} \cdot B(m - b + \ell-1,i-1,2m-j).\nonumber
\end{align}

We also state a simple lemma:
\begin{lemma}\label{lemma:p-sum-simplify}
  Let $r,s \geq 1$.
  Then
  \begin{equation*}
    p_{r,s} = C_{r+s-2} + \sum_{k=1}^{s-1} C_{k-1} \cdot p_{r,s-k}\..
  \end{equation*}
\end{lemma}
\begin{proof}
  This follows from expanding each $p_{r,s-k}$ term and changing the order of summation.
\end{proof}

\subsection{Proof of Main Lemma}
\label{pf:proof-main-theorem-enumerative}
  Suppose our formula (\ref{eq:main-lemma-equation}) holds for all $B(n,p,q)$ satisfying $n<m$, and $p,q \in \{1, \dots ,2n\}$ with $p \leq q \leq 2n-p$.
  To show that the result holds for $B(m,i,j)$, we prove it separately for each of the four choices on the parity of $i$ and $j$.
  Let $i=2a$ and $j=2b$.
  In this case, we have
  \begin{equation*}
    a = \left\lceil \frac{i}{2} \right\rceil, \qquad b = \left\lfloor \frac{j}{2} \right\rfloor, \quad \text{and} \quad h = \left\lfloor \frac{(j-i-1)}{2} \right\rfloor = b - a -1\..
  \end{equation*}
  Thus, the triple sum in equation (\ref{eq:triple-sum-reindexed}) becomes
  \begin{align*}
    B(m,2a,2b) =& \sum_{k=1}^{a-1} C_{k-1} \cdot B\bigl(m-k,2(a-k), 2b\bigr)\nonumber\\
    &+ \sum_{\ell=1}^{a} C_{b-\ell} \cdot B\bigl(m-b+\ell-1, 2(\ell-1)+1, 2(m-b-a+\ell)\bigr)\nonumber\\
    &+ \sum_{\ell=a+1}^{b} C_{b-\ell} \cdot B\bigl(m-b+\ell-1, 2a-1, 2(m-b)\bigr)\..
  \end{align*}
  By induction, this gives
  \begin{equation} \label{eq:triple-sum-even-even-inductive}
    \begin{aligned}
      B(m,2a,2b) =& \sum_{k=1}^{a-1} C_{k-1} \cdot \sum_{c=1}^{a-k-1}\sum_{d=1}^{a-k-c} p_{c,d} \cdot C_{b-c} \cdot C_{m-k-b-d} \\
      &+ \sum_{\ell=1}^{a} C_{b-\ell} \cdot \sum_{c=1}^{\ell - 2}\sum_{d=1}^{\ell - 1 - c} p_{c,d} \cdot C_{m-b-a+\ell - c} \cdot C_{a-1-d} \\
      &+ \sum_{\ell=a+1}^{b} C_{b-\ell} \cdot \sum_{c=1}^{a-1}\sum_{d=1}^{a-1-c} p_{c,d} \cdot C_{m-b-c} \cdot C_{\ell - 1 - d}\,.
    \end{aligned}
  \end{equation}

  For a fixed $r$ and $s$, let $c_1,c_2,c_3$ be the coefficients on $C_{b-r}\cdot C_{m-b-s}$ in the first, second and third sums in the above equation.
  Looking at the first sum we see that $c_1$ is only nonzero when $r \leq a-s$.
  In such a case, by considering $c = r$ and $k+d = s$ we see that the coefficient is given by
  \begin{equation*}
    c_1 = \sum_{k=1}^{s-1}p_{r,s-k} \cdot C_{k-1} = p_{r,s} - C_{r+s-2}\.,
  \end{equation*}
  with equality following from Lemma \ref{lemma:p-sum-simplify}.

  For the second sum, we consider $\ell = r$ and $a+c-r = s$.
  Since $c \geq 1$, the second sum contributes nothing for the cases when $r+s \leq a$.
  Thus for $r+s \leq a$ we have $c_2 = 0$, and for $r+s > a$ we instead get
  \begin{equation*}
    c_2 = \sum_{j=1}^{a-s} p_{r+s-a,j}\cdot C_{a-j-1}\..
  \end{equation*}

  Our third sum requires some manipulation.
  We rewrite the third sum in (\ref{eq:triple-sum-even-even-inductive}) as
  \begin{equation*}
    \sum_{c=1}^{a-1} \sum_{d=1}^{a-1-c} p_{c,d} \cdot C_{m-b-c} \cdot \left( \sum_{\ell=a+1}^{b} C_{b-\ell} \cdot C_{\ell - 1 - d} \right),
  \end{equation*}
  and here the inner sum can be re-expressed as a subtraction, through the standard Catalan convolution:
  \begin{equation*}
    \sum_{c=1}^{a-1} \sum_{d=1}^{a-1-c}  p_{c,d} \cdot C_{m-b-c} \cdot (C_{b-d} - C_{b-d-1}C_{0} - \ldots - C_{b-a}C_{a-d-1})\..
  \end{equation*}
  If we consider those summands where $c = s$ we obtain that for $r+s \leq a$ we have $c_3 = C_{r+s-2}$, and otherwise
  \begin{equation*}
    c_3 = -\left( \displaystyle\sum_{j=1}^{a-s} p_{s,j}\cdot C_{r-j-1} \right).
  \end{equation*}

  Combining these terms, we see that
  \begin{equation*}
    c_1 + c_2 + c_3 = \left\{
      \begin{array}{ll}
        p_{r,s} & \text{for $r + s \leq a$}\ts, \\
        [1em]
        \displaystyle\sum_{j=1}^{a-s} \left( p_{r+s-a,j}\cdot C_{a-j-1} - p_{s,j}\cdot C_{r-j-1} \right) & \text{for $r + s \geq a$}\ts.
      \end{array}
    \right.
  \end{equation*}
  Expand and rearrange each $p_{*,*}$ term in this latter sum, and observe that it is equal to~$0$.
  This completes the proof for the case when~$i$ and~$j$ are both even.
  The other three cases are similar and their proof is omitted.\footnote{Full proofs
  of these cases will appear in the first author's Ph.D.~thesis~\cite{Dok}.} \ $\sq$

\bigskip

\section{Asymptotic estimates}\label{sec:asymptotics}

\subsection{Preliminaries}
Let $P(m,i,j)$ be the probability that $\sigma \in \mathcal{B}_{2m}$ satisfies $\sigma(i) = j$.
Then $P(m,i,j) = B(m,i,j)/C_m$. 

Throughout the section, it will be convenient to pretend that $2 \alpha m, 2 \beta m$ are integers.
Further, we assume that $2 \alpha m, 2 \beta m$ are even, in order to restrict to one version of the formulas
in the equation~(\ref{eq:main-lemma-equation}).
Therefore, we write:
\begin{equation}\label{eq:probability}
  P(m,2 \alpha m, 2 \beta m) \, = \.\sum_{r=1}^{\alpha m - 1} \sum_{s=1}^{\alpha m - r} \frac{p_{r,s} \cdot C_{\beta m - r} \cdot C_{m -\beta m - s}}{C_m}\,. \tag{$**$}
\end{equation}

\begin{theorem}\label{t:spikes}
  We have the following corner probabilities:
  \begin{align*}
    P(m,1,1)\ts, \. P(m,1,2m-1)\ts, \. P(m,2,2m-1) \, & \rightarrow \, \frac{1}{4}\.,\\
    P(m,1,2) \. = \. P(m,1,2m) \. = \. P(m,2,2) \. &= \. 0\.,
  \end{align*}
  and the remaining corner cases are given by symmetries.
\end{theorem}
\begin{proof}
  It follows at once from the formulas in the Main Lemma that the probabilities in these cases are either~$0$,
  or $C_{m-1}/C_m \rightarrow 1/4$\ts.
\end{proof}

Before proving Theorem \ref{thm:asymptotic}, we will need to take a closer look at the partial Catalan convolutions $p_{r,s}$.
Note that $C_{r+s-2} \leq p_{r,s} \leq C_{r+s-1}$.
In fact, the following stronger result holds.
\begin{lemma}\label{lemma:partial-convo-approx}
  Fix $\rho \in (2/3, 1)$.
  Suppose that $k$ satisfies
  \begin{equation*}
    n^\rho < k < n - n^\rho\ts.
  \end{equation*}
  Then, for every $\varepsilon > 0$,
  \begin{equation*}
    (1-\varepsilon) C_{n-1} \leq 2 \cdot p_{n-k,k} \leq (1+\varepsilon) C_{n-1}\.,
  \end{equation*}
  for all $n$ large enough.
\end{lemma}
Letting $k=s, r = n-k-1$ gives this above inequality.
We omit the proof.


\subsection{Partitioning the summation}
Observe that the double sum in equation (\ref{eq:probability}) is over a
triangular set of points $\textsc{T}$, with $(r,s)\in T$ satisfying
\begin{equation*}
  1 \leq r \leq \alpha m - 1 \qquad \text{and} \quad 1 \leq s \leq \alpha m - r\..
\end{equation*}
We partition $\textsc{T}$ into the four sets $\textsc{A},\textsc{B},\textsc{C}$ and $\textsc{D}$ as follows (see Figure \ref{fig:tri-region}).
\begin{align*}
  \textsc{A} &= \{(r,s) : (r+s) \leq \alpha m^\delta\} \\
  \textsc{B} &= T \setminus (A \cup C \cup D) \\
  \textsc{C} &= \{(r,s) : s \leq (r+s)^\rho \} \\
  \textsc{D} &= \{(r,s) : s \geq r+s - (r+s)^\rho \} \\
\end{align*}
We show that the sum of terms in~(\ref{eq:probability}) corresponding
to each of the sets $\textsc{A},\textsc{C}$ and $\textsc{D}$ are dominated
by the those from set~$\textsc{B}$.
This will allow us to apply Lemma~\ref{lemma:partial-convo-approx}
and approximation in~(\ref{eq:catalan-asymptotic}) to terms whose indices
lie in set~$\textsc{B}$.
\begin{figure}[h]
  \centering
  \begin{tikzpicture}[domain=1:3,smooth,variable=\x]
    \draw[thick] (0,0)--(0,3)--(3,0)--(0,0);

    \draw[thick] (0,0)--(0,1)--(1,0)--(0,0);

    \draw[->]    (0,3) -- (0,3.5);
    \draw[->]    (3,0) -- (4,0);

    \draw (4,0)[right] node {$r$};
    \draw (0,3.5)[above] node {$s$};

    \draw (0,3)[left]    node {$\alpha m -1$};
    \draw (3,0)[below]   node {$\alpha m -1$};

    \draw (0,1)[left]    node {$\alpha m^\delta$};
    \draw (1,0)[below]   node {$\alpha m^\delta$};

    \draw (0.3,0.3)      node {$\textsc{A}$};
    \draw (1,1)          node {$\textsc{B}$};
    \draw (2,0.23)       node {$\textsc{C}$};
    \draw (0.2,2)        node {$\textsc{D}$};
    \draw (2.7,1)        node {$\textsc{T}$};

    \draw (0,0.1)[left]  node {$1$};
    \draw (0.1,0)[below] node {$1$};

    \draw[thick] plot ({\x - 0.4*sqrt((2*\x-1.5)/2)}, {0.4*(sqrt((2*\x-1.5)/2))});
    \draw[thick] (1,0)--(3,0)--(3-0.5865,0.5865);
    \begin{scope}[xscale=-1,rotate=90]
      \draw[thick] plot ({\x - 0.4*sqrt((2*\x-1.5)/2)}, {0.4*(sqrt((2*\x-1.5)/2))});
      \draw[thick] (1,0)--(3,0)--(3-0.5865,0.5865);
    \end{scope}

  \end{tikzpicture}
  \caption{The summation points, partitioned into four sets.}
  \label{fig:tri-region}
\end{figure}
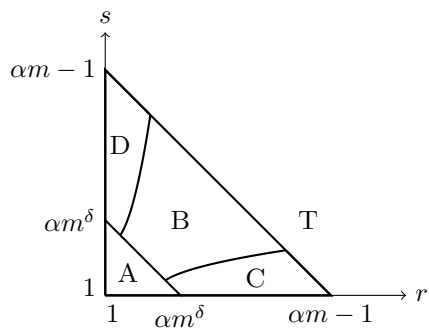

We reduce the problem of summation over $\textsc{T}$ to an integral
over the right triangle $\Delta$, given by $x,y \geq 0$ and $x+y < \alpha$.
For that we similarly define the regions $\textbf{A},\textbf{B},\textbf{C}$ and $\textbf{D}$ to be all $(x,y)$ in $\Delta$ so that $(\lceil mx \rceil , \lceil my \rceil)$ is a point in $\textsc{A},\textsc{B},\textsc{C}$ or $\textsc{D}$, respectively.

\subsection{Estimates for set ${\normalfont \textsc{A}}$}\label{sec:set-A-is-small}
Each of the three sets $\textsc{A},\textsc{C}$ and $\textsc{D}$ contains a relatively small number of terms.
We will show that the sum of terms in (\ref{eq:probability}) corresponding to these sets grows slowly, so that their contribution is dwarfed by that of set $\textsc{B}$.
\begin{lemma}\label{lemma:A-is-small}
  Let $\delta < 1/4$.  
  Then the sum of terms in $($\ref{eq:probability}$)$ corresponding to ${\normalfont \textsc{A}}$ is $o \! \left( \frac{1}{m} \right)$.
\end{lemma}
\begin{proof}
  The largest individual term in (\ref{eq:probability}) is obtained when $r=s=1$. We get:
  \begin{equation*}
    \frac{2 C_{\beta m -1} C_{m -  \beta m   - 1}}{C_m} \,\sim \, \frac{1}{\beta(1-\beta)8\sqrt{\pi}} \cdot \frac{1}{m^{3/2}}\,.
  \end{equation*}
  We bound each of the summands from $\textsc{A}$ by this, so that their total is at most
  \begin{align*}
    & |\textsc{A}| \cdot \frac{2 C_{ \beta m   -1} C_{m -  \beta m   - 1}}{C_m}
    \, = \, (\alpha m^\delta)^2 \cdot \frac{C_{\beta m   -1} C_{m - \beta m - 1}}{C_m} \\
    & \qquad \ \ \sim \, \frac{1}{8\ts \sqrt{\pi}\ts\beta(1-\beta)} \cdot \frac{(\alpha m^\delta)^2}{m^{3/2}}
    \, = \, \frac{\alpha^2}{8\sqrt{\pi}\beta(1-\beta)} \cdot \frac{1}{m^{3/2 - 2\delta}}
    \, = \, o\left( \frac{1}{m} \right),
  \end{align*}
  where the last equality follows from $\delta < 1/4$\ts.
\end{proof}

\subsection{Estimates for sets ${\normalfont \textsc{C} \ and \ \textsc{D}}$}\label{sec:sets-CD-are-small}
\begin{lemma}\label{lemma:CD-are-small}
  Let $\rho \in (2/3, 1)$.
  Then the sum of terms in $($\ref{eq:probability}$)$ corresponding to set $\textsc{C}$ is $o \! \left( \frac{1}{m} \right)$.
  The same also holds for the set~$\textsc{D}$.
\end{lemma}
\begin{proof}  Choose a constant $M$ large enough so that
  \begin{equation*}
    C_n \, \leq \, M \cdot \frac{4^n}{(n+1)^{3/2}}\,, \qquad \text{for all} \, \ n \geq 0\..
  \end{equation*}
  Denote by $\Psum{C}$ the sum of terms in (\ref{eq:probability}) corresponding to~$\textsc{C}$, i.e.
  \begin{equation*}
    \Psum{C} \, = \,
    \sum_{(r,s)\in \textsc{C}} \frac{p_{r,s} \cdot C_{\beta m  - r} \cdot C_{m - \beta m  - s}}{C_m}
    \, \leq \.\sum_{(r,s)\in \textsc{C}} \frac{C_{r+s-1} \cdot C_{\beta m  - r} \cdot C_{m - \beta m  - s}}{C_m}\,,
  \end{equation*}
  where the last inequality follows from \ts $p_{r,s} \leq C_{r+s-1}$.
  We use formula (\ref{eq:catalan-asymptotic}) as well as our choice of~$M$ to obtain
  \begin{align*}
    \frac{4 \Psum{C}}{\sqrt{\pi} \ts M^3} \, &\leq  \,
    \sum_{(r,s)\in \textsc{C}} \frac{m^{3/2}}{[(r+s)(\beta m - r + 1)((1-\beta)m-s+1)]^{3/2}}\,.
     \end{align*}
  And as $m\rightarrow \infty$ this last sum is asymptotically equivalent to
  \begin{equation*}
    \frac{1}{m} \iint_{\textbf{C}} \frac{dx\ dy}{[(x+y) (\beta - x) (1-\beta - y)]^{3/2}}\,.
  \end{equation*}
  Now observe that the above integral converges, and that $Area(\textbf{C}) \rightarrow 0$
  as $m\rightarrow \infty$.  Thus, $\Psum{C}=o\left( \frac{1}{m} \right)$, as desired.
  The proof for $\textsc{D}$ follows verbatim from the above argument and the bound
  $p_{r,s} \leq C_{r+s-1}$.
\end{proof}

\begin{remark}
  At the beginning of this section, we assumed that $2\alpha m, 2\beta m$ are even integers.
  For the other choices of parity, the sum (\ref{eq:probability}) differs in the range of indices $(r,s)$, and, depending on the chosen parity of $2\alpha m$, the addition of a $O(m^{-3/2})$ term, coming from the first term in the summation (\ref{eq:main-lemma-equation}) from the main lemma.
  The above arguments and the following proof of Theorem \ref{thm:asymptotic} can be adapted with little difficulty to handle these differences.
\end{remark}

\medskip

\subsection{Proof of Theorem \ref{thm:asymptotic}}\label{sec:proof-of-asymptotic}
  Denote by $\Psum{B}$ the sum of terms from (\ref{eq:probability}) corresponding to set $\textsc{B}$, i.e.
  \begin{equation*}
    \Psum{B} \, = \. \sum_{(r,s) \in \textsc{B}}\frac{p_{r,s} \cdot C_{\beta m -r} \cdot C_{m - \beta m -s}}{C_m}\,.
  \end{equation*}

We need to prove that $\Psum{B}\sim \frac{1}{m} \varphi(\alpha,\beta)$.
  This follows from the construction of $\textsc{B}$.
  First, note that for $(r,s) \in \textsc{B}$, the indices $r+s,~\beta m - r$ and $(1-\beta) m - s \to \infty$
  as $m\to \infty$.
  Thus, by Lemma \ref{lemma:partial-convo-approx} we have
  \begin{equation*}
    \Psum{B} \, \sim \, \frac{1}{2} \sum_{(r,s)\in \textsc{B}} \frac{C_{r+s-1} \cdot C_{\beta m -r} \cdot C_{m - \beta m -s}}{C_m}\,.
  \end{equation*}
  From formula (\ref{eq:catalan-asymptotic}) we obtain
  \begin{align*}
    \Psum{B} \, &\sim \,\frac{1}{8\pi} \. \sum_{(r,s)\in \textsc{B}} \. \frac{m^{3/2}}{[(r+s)(\beta m - r)((1-\beta)m - s)]^{3/2}}\\
    &\sim \, \frac{1}{8\pi m} \. \iint_{\textbf{B}} \. \frac{dx\ dy}{[(x+y)(\beta-x)(1-\beta-y)]^{3/2}}\,.
  \end{align*}
  As $m\to \infty$, the region $\textbf{B}$ expands to fill the entire triangle~$\Delta$, so we obtain
  \begin{equation*}
    \Psum{B} \, \sim \, \frac{1}{8\pi m}\. \int_0^\alpha
    \int_0^{\alpha - y} \frac{dx\ dy}{[(x+y)(\beta-x)(1-\beta-y)]^{3/2}} \, = \, \frac{\varphi(\alpha,\beta)}{m}\,.
  \end{equation*}
  Finally, by Lemmas \ref{lemma:A-is-small} and \ref{lemma:CD-are-small}, the probability
  $P(m,2\alpha m, 2\beta m) \sim \Psum{B}$.    This completes the proof. \ $\sq$

\bigskip

\section{Final remarks and open problems}
\label{sec:final-remarks}

\subsection{}
There is a large literature on the asymptotic behavior of Catalan structures,
too large for us to summarize.  We refer to~\cite{FS,Odl,PW} for a comprehensive
introduction to asymptotic methods, and to~\cite{Pit} for strongly related
discrete probability results.

\subsection{}\label{s:fin-bax}
The choice of doubly alternating Baxter permutations is a part of a general
program of study of pattern-avoiding permutations. Baxter permutations are
defined to avoid two (generalized) patterns, and widely studied in the literature
(see e.g.~~\cite{ackerman-et, bonichon-et, dulucq-guibert-baxter,FFNO},
in part due to their connection to plane bipolar orientations, tilings~\cite{Korn},
and other combinatorial objects (see~\cite[$\S 2.2$]{Kit} for the introduction
and references).  As mentioned in the introduction, the alternating permutations
are classical in the own right; we refer to~\cite{Sta} for a comprehensive survey
on their role in Enumerative Combinatorics, to~\cite{DW} for a recent detailed
probabilistic study, and to~\cite[$\S 6.1$]{Kit} for connections to other
pattern-avoiding permutations.

\subsection{}\label{s:fin-3}
The study of random pattern avoiding permutations by means of the limit shape was
recently introduced by Miner and the second author~\cite{miner-pak}.  The authors
obtained highly detailed information for {\bf 123} and {\bf 213} patterns,
each a classical Catalan structure.  In both cases the limit surface is degenerate
as most permutation matrix entries concentrate around the anti-diagonal, and the
interesting behavior occurs in the micro-scale near the anti-diagonal with various
phase transitions.  We refer to~\cite{AM,ML} for strongly related results, and
some interesting extensions for larger patterns.

\subsection{}
Following the ideas in~\cite{ML} and~\cite{miner-pak}, one can ask whether for a
given (generalized) pattern there is always a limit surface, and whether
it is degenerate.  Since the set of bistochastic matrices is a compact (in fact, the
\emph{Birkhoff polytope}), the answer to the former is likely yes.  But the latter question
seems difficult and currently out of reach.

\subsection{}
The fact that the limit surface $\Phi(x,y)$ is highly symmetric and piecewise smooth
came as a surprise, as in the initial experiments for averages of all $\si \in \Bax_n$,
the graphs exhibited fewer symmetries and numerous small spikes, see Figure~\ref{fig:skipes}.
Of course, these spiked are due to the the parity differences in
formula~$(\ast)$ in the Main Lemma, and as $n$~grows, they gradually disappear.
Curiously, the asymmetry across the main diagonal persisted until very large~$n$,
suggesting a rather slow convergence in Theorem~\ref{thm:asymptotic}.

\begin{figure}[ht]
  \centering
  \subcaptionbox{}{\includegraphics[width=5.8cm]{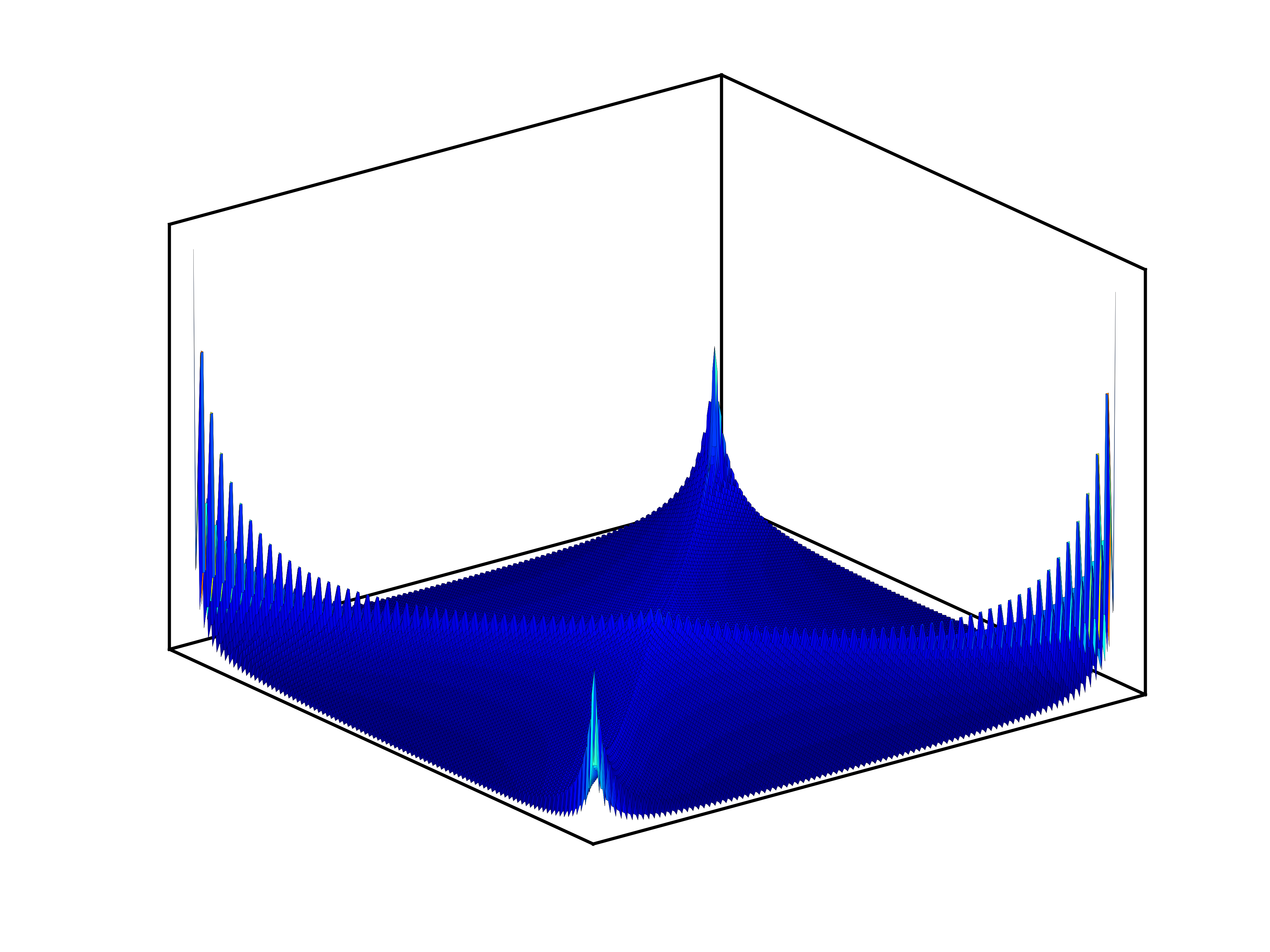}} \qquad \
  \subcaptionbox{}{\includegraphics[width=5.8cm]{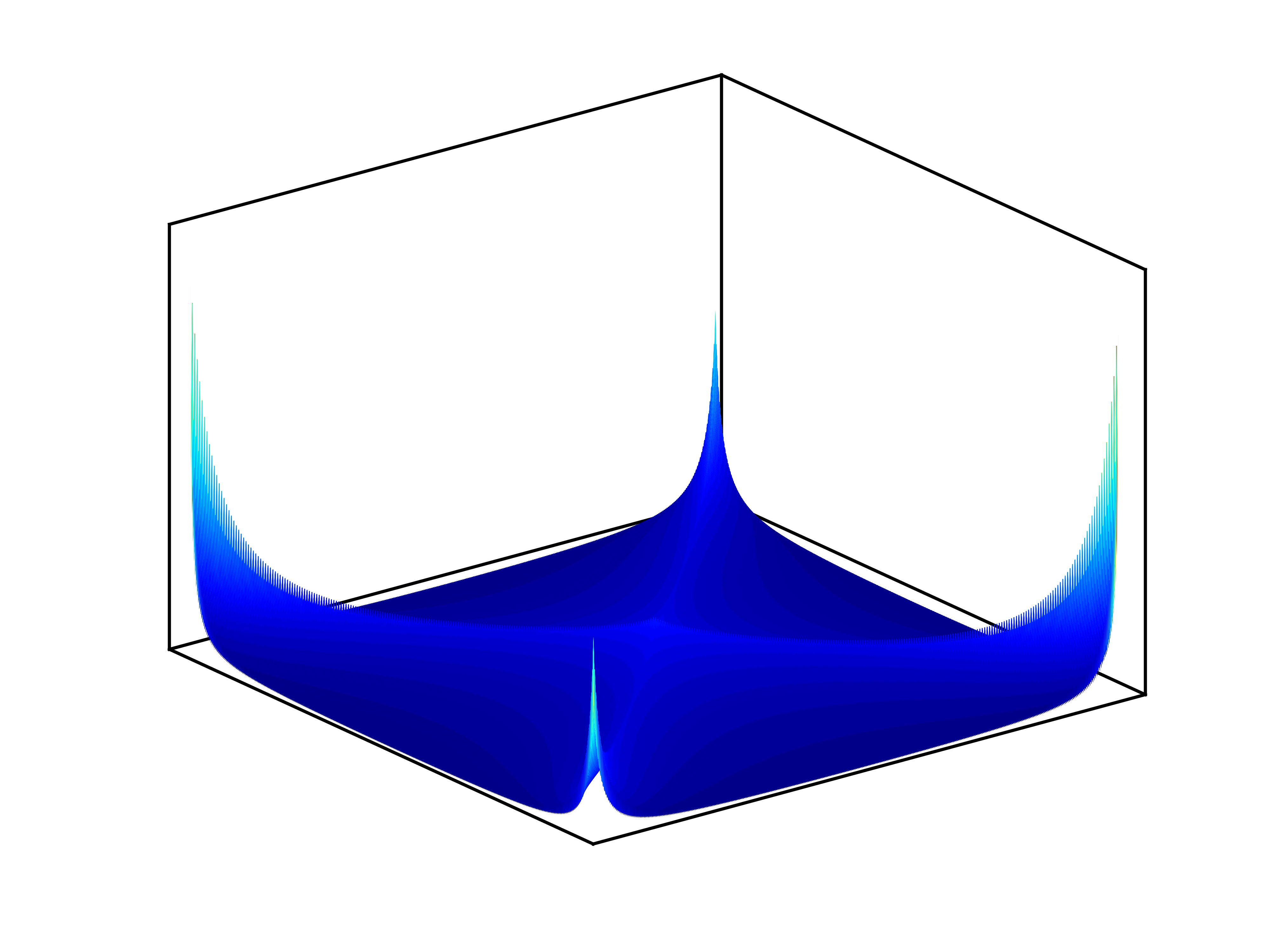}}
  \caption{Averages of all $\si\in \Bax_n$, where $n=200$ and~$800$.}
  \label{fig:skipes}
\end{figure}

To illustrate this phenomenon, note that the limit shape has an extra
degree of symmetry, which the actual numbers $B(m,i,j)$ do not have.
The following two graphs are for $\alpha = 3/40$, and the horizontal axis
is~$\beta$.  The red line is $\varphi(\alpha,\beta)$, the blue line is
$m \cdot P(m,2\alpha m, 2\beta m)$, for $m = 500$ on the left and $m = 2000$
on the right.  For the emphasis, recall that in the notation above,
we have $n=2m$.
\begin{figure}[ht]
  \centering
  \subcaptionbox{}{\includegraphics[width=5.8cm]{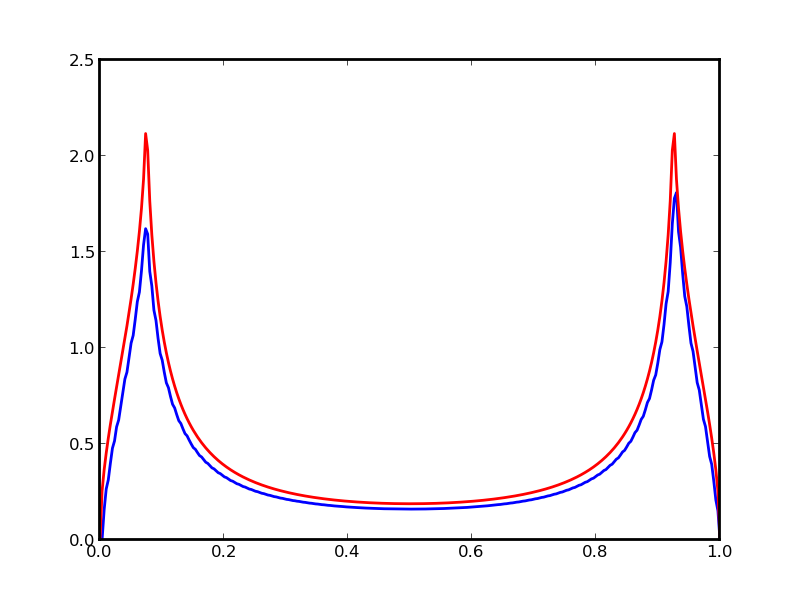}} \qquad \
   \subcaptionbox{}{\includegraphics[width=5.8cm]{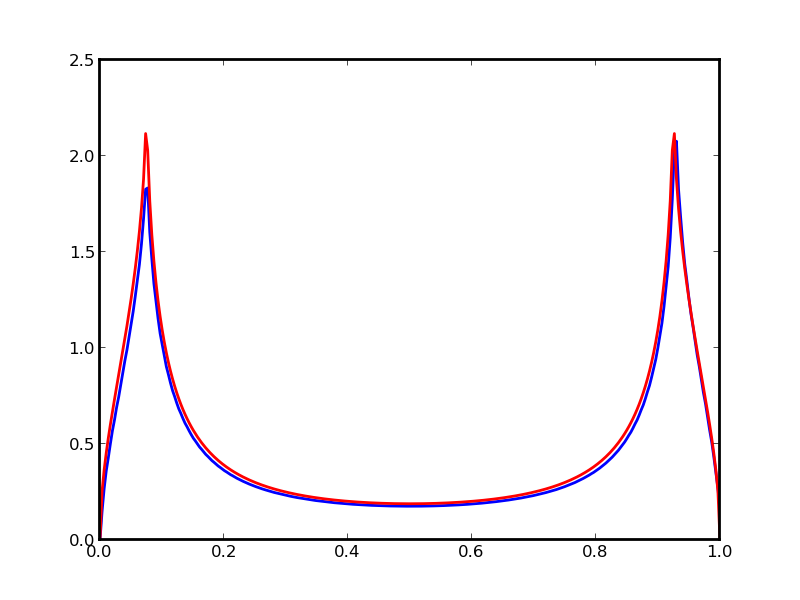}}
  \caption{Horizontal slices of the limit surface.}
  \label{fig:horiz-slices}
\end{figure}

\subsection{}
Let us note that for the case of size~3 patterns mentioned in~$\S\ref{s:fin-3}$,
the random restricted permutations do in fact resemble their limit shapes, due
to exponentially small decay of probabilities away from the anti-diagonal.
However, in the the case of random $\si \in\Bax_n$, permutations tend to
exhibit a high degree or structure, and do not resemble the limit
surface~$\Phi(x,y)$.  This suggests that computing square-to-square
correlations would an interesting problem.  While we expect this to
be doable, this goes outside the scope of our project.

\begin{figure}[ht]
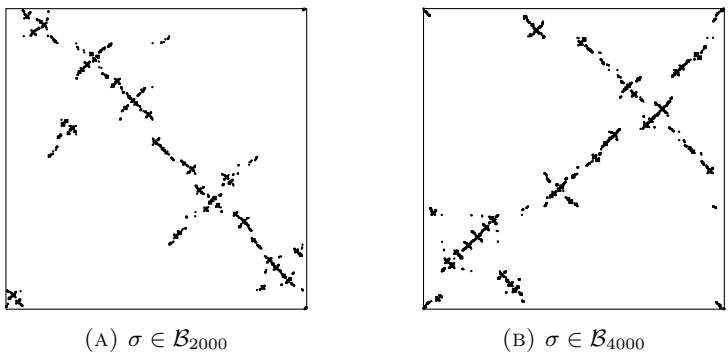

  \centering
\subcaptionbox{$\sigma \in \mathcal{B}_{2000}$}{\input{sample-m-1000.tikz}}\qquad \qquad
\subcaptionbox{$\sigma \in \mathcal{B}_{4000}$}{\input{sample-m-2000.tikz}}
  \caption{Randomly sampled permutations $\si \in \mathcal{B}_{2000}$ and $\mathcal{B}_{4000}$.}
  \label{fig:rand}
\end{figure}

\subsection{}
It would be interesting to compute the limit shape of random Baxter permutations.
While they can be sampled exactly uniformly (see e.g.~\cite{FFNO,Vie}), the
underlying bijection does not seem to give any useful formulas.  In fact,
a special effort is required to sample beyond $n=20$.  Still, the
apparent connection to $\Phi(x,y)$ is undeniable (see Figure~\ref{fig:baxter-alt}).

\begin{figure}[ht]
  \centering
   \subcaptionbox{}{\includegraphics[width=5.8cm]{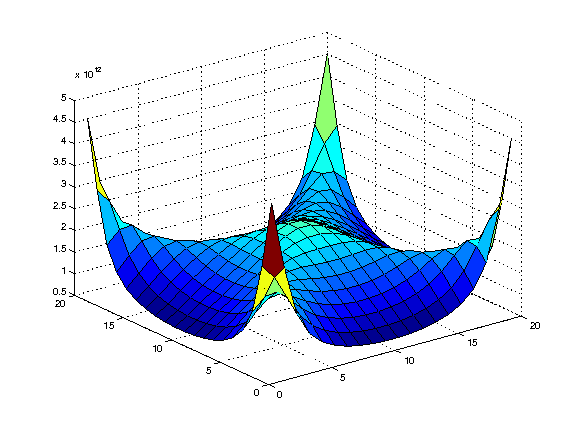}} \qquad \
   \subcaptionbox{}{\includegraphics[width=5.8cm]{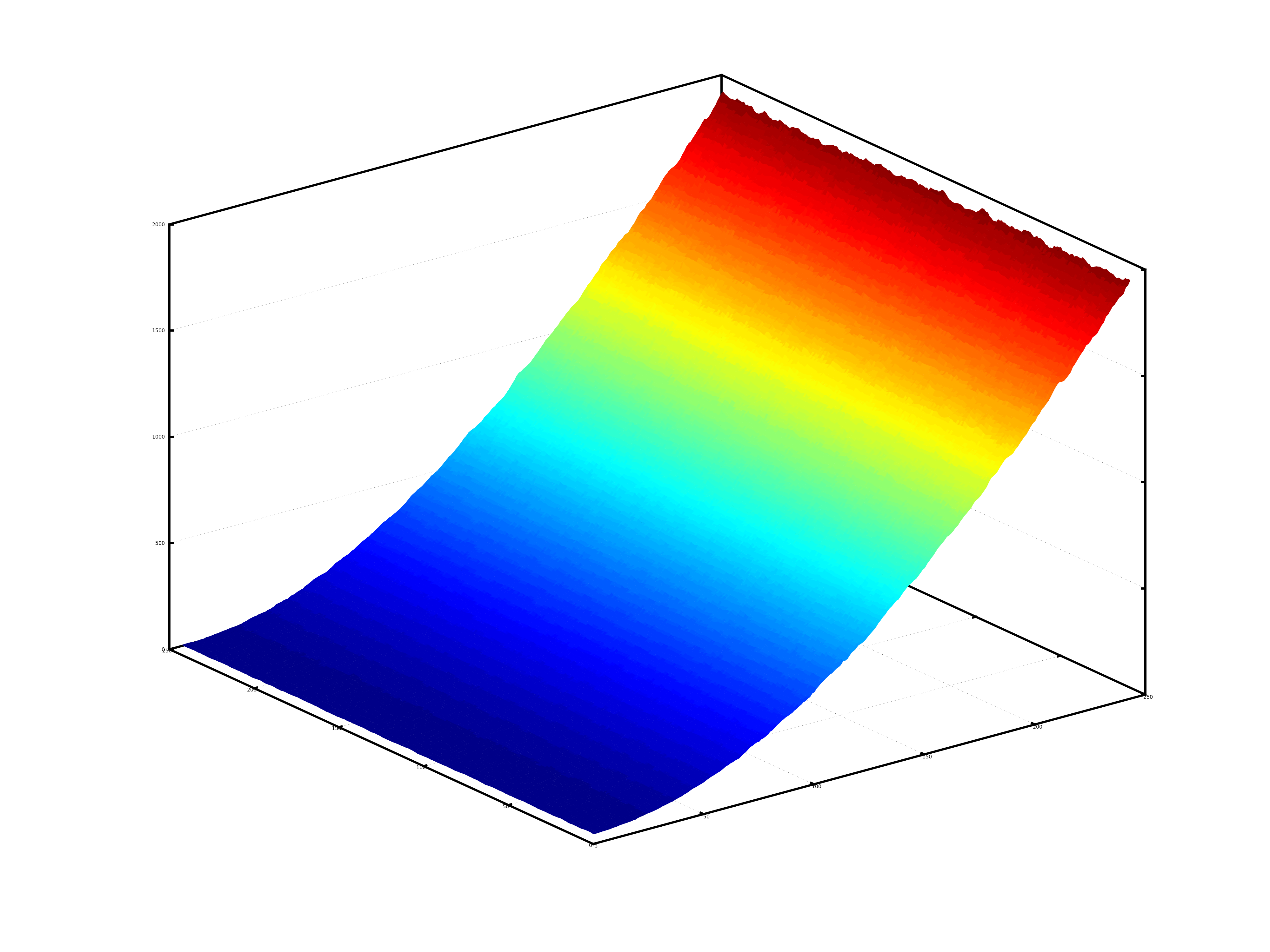}}
  \caption{Averages of Baxter permutations ($n=20$) and alternating permutations $(n=500)$.}
  \label{fig:baxter-alt}
\end{figure}

A possible explanation lies in the ``flat structure'' of alternating
permutations.  As evident from Figure~\ref{fig:baxter-alt} for~$n=500$,
there seem to be a limit surface of alternating permutations, with no spikes
except at the boundary.  In fact, the existence of such a flat surface follows
from the asymptotics of Entringer numbers given in~\cite{DW}.  This suggests
that all spikes in $\Phi(\cdot,\cdot)$ come from the ``Baxter condition''.

\vskip.7cm

\noindent
\textbf{Acknowledgments:} \. The authors are grateful to
Stephen DeSalvo, Scott Garrabrant, Neal Madras and Sam Miner
for useful remarks and help with the references.
The second author was partially supported by the~NSF.


\end{document}